\documentclass{svproc}
\usepackage{marvosym}
\usepackage{graphicx}

\usepackage{url}
\usepackage{hyperref}

\providecommand{\doi}[1]{doi:\discretionary{}{}{}#1}

\usepackage{algorithm}
\usepackage{algpseudocode}
\usepackage{amsmath}
\usepackage{amssymb}
\usepackage{mathtools}

\newcommand*\ceq{\mathop{\vcenter{\hbox{:}}{=}}}

\newcommand\doubleplus{\ensuremath{\mathbin{+\mkern-10mu+}}}
\newcommand*\reduceop{\mathop{\mathrlap{\bigcirc}\downarrow}}
\newcommand{\algrule}[1][.2pt]{\par\vskip.01\baselineskip\hrule height #1\par\vskip.01\baselineskip}

\def\orcidID#1{\unskip$^{[#1]}$}
\def\letter{$^{\textrm{(\Letter)}}$}

\begin{document}
\mainmatter              % start of a contribution
\title{Visualizing Multidimensional Linear Programming Problems}
\titlerunning{Visualizing LP Problems}  % abbreviated title (for running head)
\author{Nikolay~A.~Olkhovsky \and Leonid~B.~Sokolinsky\letter\orcidID{0000-0001-9997-3918}}
\authorrunning{N.A. Olkhovsky and L.B. Sokolinsky} % abbreviated author list (for running head)
\tocauthor{Nikolay A. Olkhovsky and Leonid B. Sokolinsky}
\institute{
	South Ural State University {(National Research University)} \\
	76, Lenin prospekt, Chelyabinsk, Russia, 454080\\
	\email{olkhovskiiNA@susu.ru}, \email{leonid.sokolinsky@susu.ru}
}

\maketitle              % typeset the title of the contribution

\begin{abstract}
The article proposes an n-dimensional mathematical model of the visual representation of a linear programming problem. This model makes it possible to use artificial neural networks to solve multidimensional linear optimization problems, the feasible region of which is a bounded non-empty set. To visualize the linear programming problem, an objective hyperplane is introduced, the orientation of which is determined by the gradient of the linear objective function: the gradient is the normal to the objective hyperplane. In the case of searching a maximum, the objective hyperplane is positioned in such a way that the value of the objective function at all its points exceeds the value of the objective function at all points of the feasible region, which is a bounded convex polytope. For an arbitrary point of the objective hyperplane, the objective projection onto the polytope is determined: the closer the objective projection point is to the objective hyperplane, the greater the value of the objective function at this point. Based on the objective hyperplane, a finite regular set of points is constructed, called the receptive field. Using objective projections, an image of a polytope is constructed. This image includes the distances from the receptive points to the corresponding points of the polytope surface. Based on the proposed model, parallel algorithms for visualizing a linear programming problem are constructed. An analytical estimation of its scalability is performed. Information about the software implementation and the results of large-scale computational experiments confirming the efficiency of the proposed approaches are presented.
\keywords{Linear programming $\cdot$ Multydimensional visualization $\cdot$ Mathematical model $\cdot$ Parallel algorithm $\cdot$ BSF-skeleton}
\end{abstract}
\section{Introduction}
The rapid development of big data technologies~\cite{Jagadish2014,Hartung2018} has led to the emergence of mathematical optimization models in the form of large-scale linear programming (LP) problems~\cite{Sokolinskaya2017a}. Such problems arise in industry, economics, logistics, statistics, quantum physics, and other fields~\cite{Chung2015,Gondzio2014,Sodhi2005,Brogaard2014,Sokolinskaya2017b}. In many cases, the conventional software is not able to handle such large-scale LP problems in an acceptable time~\cite{Bixby2002}. At the same time, in the nearest future, exascale supercomputers will appear~\cite{Dongarra2019},  which are potentially capable of solving such problems. In accordance with this, the issue of developing new effective methods for solving large-scale LP problems using exascale supercomputing systems is urgent.

Until now, the class of algorithms proposed and developed by Dantzig based on the simplex method~\cite{Dantzig1998} is one of the most common ways to solve the LP problems. The simplex method is effective for solving a large class of LP problems. However, the simplex method has some fundamental features that limit its applicability to large LP problems. First, in the worst case, the simplex method traverses all the vertices of the simplex, which results in exponential time complexity~\cite{Zadeh1973}. Second, in most cases, the simplex method successfully solves LP problems containing up to 50,000 variables. However, a loss of precision is observed when the simplex method is used for solving large LP problems. Such a loss of precision cannot be compensated even by applying such computational intensive procedures as ``affine scaling'' or ``iterative refinement''~\cite{Tolla2014}. Third, the simplex method does not scale well on multiprocessor systems with distributed memory. A lot of attempts to parallelize the simplex method were made, but they all failed~\cite{Mamalis2015}. In~\cite{Karmarkar1984}, Karmarkar proposed the inner point method having polynomial time complexity in all cases. This method effectively solves problems with millions of variables and millions of constraints. Unlike the simplex method, the inner point method is self-correcting. Therefore, it is robust to the loss of precision in computations. The drawbacks of the interior point method are as follows. First, the interior point method requires careful tuning of its parameters. Second, this method needs a known point that belongs to the feasible region of the LP problem to start calculations. Finding such an interior point can be reduced to solving an additional LP problem. An alternative is the iterative projection-type methods~\cite{Sokolinskaya2018,Sokolinskaya2018a,Sokolinsky2020b}, which are also self-correcting. Third, like the simplex method, the inner point method does not scale well on multiprocessor systems with distributed memory. Several attempts at effective parallelization for particular cases have been made  (see, for example,~\cite{Hafsteinsson1994,Karypis1994}). However, it was not possible to make efficient parallelization for the general case. In accordance with this, the research directions related to the development of new scalable methods for solving LP problems are urgent.

A possible efficient alternative to the conventional  methods of LP is optimization methods based on neural network models. Artificial neural networks~\cite{Prieto2016,Schmidhuber2015} are one of the most promising and rapidly developing areas of modern information technology. Neural networks are a universal tool capable of solving problems in almost all areas. The most impressive successes have been achieved in image recognition and analysis using convolutional neural networks~\cite{LeCun2015}. However, in the scientific periodicals, there are almost no works devoted to the use of convolutional neural networks for solving linear optimization problems~\cite{Lachhwani2020}. The reason is that convolutional neural networks are focused on image processing, but there are no works on the visual representation of multidimensional linear programming problems in the scientific literature. Thus, the issue of developing new neural network models and methods focused on linear optimization remains open.

In this paper, we have tried to develop a n-dimensional mathematical model of the visual representation of the LP problem. This model allows us to employ the technique of artificial neural networks to solve multidimensional linear optimization problems, the feasible region of which is a bounded nonempty set. The visualization method based on the described model has a high computational complexity. For this reason, we propose its implementation as a parallel algorithm designed for cluster computing systems. The rest of the paper is organized as follows. Section~\ref{sec:model} is devoted to the design of a mathematical model of the visual representation of multidimensional LP problems. Section~\ref{sec:parallel_algorithm} describes the implementation of the proposed visualization method as a parallel algorithm and provides an analytical estimation of its scalability. Section~\ref{sec:computational_experiments} presents information about the software implementation of the described parallel algorithm and discusses the results of large-scale computational experiments on a cluster computing system. Section~\ref{sec:conclusion} summarizes the obtained results and provides directions for further research.
\section{Mathematical Model of the LP Visual Representation}\label{sec:model}
The linear optimization problem can be stated as follows
\begin{equation}\label{eq:LP-problem}
	\bar x = \arg\max \left\{ \left\langle c,x \right\rangle \middle| Ax \leqslant b, x \in \mathbb{R}^n  \right\},
\end{equation}	
where $c,b \in \mathbb{R}^n$, $A \in \mathbb{R}^{m\times n}$, and $c\ne \mathbf{0}$. Here and below, $\left\langle \cdot\;,\cdot\right\rangle $ stands for the dot product of vectors. We assume that 	constraint $x \geqslant \mathbf{0}$ is also included in the system $Ax \leqslant b$ in the form of the following inequalities:
\begin{equation*}
	\begin{array}{*{20}{c}}
	{ - {x_1}}& + &0& + & \cdots & \cdots & \cdots & + &0& \leqslant &{0;} \\
	0& - &{{x_2}}& + &0& + & \cdots & + &0& \leqslant &{0;} \\
	\cdots & \cdots & \cdots & \cdots & \cdots & \cdots & \cdots & \cdots & \cdots & \cdots & \cdots  \\
	0& + & \cdots & \cdots & \cdots & + &0& - &{{x_n}}& \leqslant &{0.}
	\end{array}
\end{equation*}
The vector $c$ is the gradient of the linear objective function
\begin{equation}\label{eq:objective_function}
	f(x)=c_1 x_1+\ldots+c_n x_n.
\end{equation}
Let $M$ denote the feasible region of problem (1):
\begin{equation}\label{eq:M}
	M = \left\{x \in \mathbb{R}^n\middle| Ax \leqslant b\right\}.
\end{equation}
We assume from now on that $M$ is a nonempty bounded set. This means that $M$ is a convex closed polytope in space~$\mathbb{R}^n$, and the solution set of problem (1) is not empty.

Let $\tilde a_i\in\mathbb{R}^n$ be a vector formed by the elements of the $i$th row of the matrix~$A$. Then, the matrix inequality $Ax \leqslant b$ is represented as a system of the inequalities
\begin{equation}\label{eq:inequalities}
	\left\langle\tilde a_i,x\right\rangle \leqslant b_i,i=1,\ldots,m.
\end{equation}
We assume from now on that 
\begin{equation}\label{eq:a_i_ne_0}
	\tilde a_i\ne \mathbf{0}.
\end{equation}
for all $i=1,\ldots,m$. Let us denote by $H_i$ the hyperplane defined by the equation
\begin{equation}\label{eq:H_i}
	\left\langle \tilde a_i,x \right\rangle = b_i \; (1\leqslant i \leqslant m).
\end{equation}
Thus, 
\begin{equation}\label{eq:hyperplane}
	H_i=\left\lbrace x\in\mathbb{R}^n \middle|\left\langle \tilde a_i,x \right\rangle = b_i \right\rbrace.
\end{equation}
\begin{definition}\label{def:hyperplane_and_halfspace}
	The half-space $H_{i}^+$ generated by the hyperplane $H_i$ is the half-space defined by the equation
	\begin{equation}\label{eq:halfspace}
		H_{i}^+=\left\lbrace x\in\mathbb{R}^n \middle|\left\langle \tilde a_i,x \right\rangle \leqslant b_i \right\rbrace.
	\end{equation}
\end{definition}
From now on, we assume that problem~\eqref{eq:LP-problem} is non-degenerate, that is
\begin{equation}\label{eq:Hi<>Hj}
	\forall i\ne j\,:H_i\ne H_j\,\left( i,j\in \left\lbrace 1,\ldots,m\right\rbrace \right) .
\end{equation}
\begin{definition}\label{def:recessive} 
	The half-space $H_{i}^+$ generated by the hyperplane $H_i$ is recessive with respect to vector~$c$ if
	\begin{equation}\label{eq:recessive_halfspace}
		\forall x\in H_i,\forall \lambda\in\mathbb{R}_{> 0}\,:x-\lambda c\in H_i^+ \wedge x-\lambda c\notin H_i.
	\end{equation}
\end{definition}
In other words, the ray coming from the hyperplane $H_{i}$ in the direction opposite to the vector $c$ lies completely in $H_{i}^+$, but not in $H_{i}$.
\begin{proposition}\label{prp:recessive_halfspace_condition}
	The necessary and sufficient condition for the recessivity of the half-space $H_{i}^+$ with respect to the vector $c$ is the condition
	\begin{equation}\label{eq:recessive_halfspace_condition}
		\left\langle \tilde a_i,c\right\rangle>0.
	\end{equation}
\end{proposition}
\begin{proof}
	Let us prove the necessity first. Let the condition~\eqref{eq:recessive_halfspace} hold. Equation~\eqref{eq:hyperplane} implies
	\begin{equation}\label{eq:x_in_H_i}
		x=\frac{b_i\tilde a_i}{\|\tilde a_i\|^2}\in H_i.
	\end{equation}
By virtue of~\eqref{eq:a_i_ne_0},
	\begin{equation}\label{eq:lambda_in_R_>0}
	\lambda=\frac{1}{\|\tilde a_i\|^2}\in\mathbb{R}_{> 0}.
\end{equation}
Comparing~\eqref{eq:recessive_halfspace} with~\eqref{eq:x_in_H_i} and \eqref{eq:lambda_in_R_>0}, we obtain
\begin{eqnarray*}
	&&\frac{b_i\tilde a_i}{\|\tilde a_i\|^2}-\frac{1}{\|\tilde a_i\|^2}c \in H_i^+;\\
	&&\frac{b_i\tilde a_i}{\|\tilde a_i\|^2}-\frac{1}{\|\tilde a_i\|^2}c \notin H_i.\\		
\end{eqnarray*}
In view of~\eqref{eq:hyperplane} and~\eqref{eq:halfspace}, this implies
	\begin{eqnarray}\label{eq:inequality}
	&\left\langle\tilde a_i, \frac{b_i\tilde a_i}{\|\tilde a_i\|^2}-\frac{1}{\|\tilde a_i\|^2}c\right\rangle <b_i.
\end{eqnarray}
Using simple algebraic transformations of inequality~\eqref{eq:inequality}, we obtain~\eqref{eq:recessive_halfspace_condition}. Thus, the necessity is proved.

Let us prove the sufficiency by contradiction. Assume that~\eqref{eq:recessive_halfspace_condition} holds, and there are $x\in H_i$ and $\lambda>0$ such that 
	\begin{equation*}%\label{key}
	x-\lambda c\notin H_i^+ \vee x-\lambda c\in H_i.
\end{equation*}
In accordance with~\eqref{eq:hyperplane} and \eqref{eq:halfspace}, this implies
	\begin{equation*}%\label{key}
	\left\langle \tilde a_i, x-\lambda c\right\rangle \geqslant b_i
\end{equation*}
that is equivalent to
	\begin{equation*}%\label{key}
	\left\langle \tilde a_i, x\right\rangle - \lambda\left\langle \tilde a_i,  c\right\rangle \geqslant b_i.
\end{equation*}
Since $\lambda>0$, it follows from~\eqref{eq:recessive_halfspace_condition} that
	\begin{equation*}%\label{key}
	\left\langle \tilde a_i, x\right\rangle >b_i,
\end{equation*}
but this contradicts our assumption that $x\in H_i$. \qed
\end{proof}	
\begin{definition}\label{def:object_halfspace}
	Fix a point $z\in\mathbb{R}^n$ such that the half-space 
	\begin{equation}\label{eq:object_halfspace}
		H_c^+=\left\{x\in\mathbb{R}^n \middle| \left\langle c,x-z\right\rangle \leqslant 0 \right\}
	\end{equation}
	includes the polytope $M$:
	\begin{equation*}%\label{key}
		M \subset H_c^+.
	\end{equation*}
	In this case, we will call the half-space $H_c^+$ the objective half-space, and the hyperplane $H_c$, defined by the equation
		\begin{equation}\label{eq:H_c}
		H_c=\left\{x\in\mathbb{R}^n \middle| \left\langle c,x-z\right\rangle = 0 \right\},
	\end{equation}
	the objective hyperplane.
\end{definition}
Denote by $\pi_c(x)$ the \emph{orthogonal projection} of point $x$ onto the objective hyperplane $H_c$:
\begin{equation}\label{eq:ortogonal_projection}
	\pi_c(x)=x-\frac{\left\langle c,x-z\right\rangle}{\|c\|^2}c.
\end{equation}
Here, $\|\cdot\|$ stands for the Euclidean norm. Define \emph{distance} $\rho_c(x)$ from $x\in H_c^+$ to the objective hyperplane ~$H_c$ as follows:
\begin{equation}\label{eq:rho_c}
	\rho_c(x)=\|\pi_c(x)-x\|.
\end{equation}
Comparing~\eqref{eq:object_halfspace}, \eqref{eq:ortogonal_projection} and \eqref{eq:rho_c}, we find that, in this case, the distance $\rho_c(x)$ can be calculated as follows:
\begin{equation}\label{eq:distance}
	\rho_c(x)=\frac{\left\langle c,z-x\right\rangle}{\|c\|}.
\end{equation}
The following Proposition~\ref{prp:coloring} holds.
\begin{proposition}\label{prp:coloring} 
	For all $x,y \in H_c^+$,
	\[\rho_c(x) \leqslant \rho_c(y) \Leftrightarrow \left\langle c,x\right\rangle \geqslant \left\langle c,y\right\rangle.\]
\end{proposition}
\begin{proof}
	Equation~\eqref{eq:distance} implies that
	\begin{eqnarray*}
		&&\rho_c(x) \leqslant \rho_c(y) \Leftrightarrow \frac{\left\langle c,z-x\right\rangle}{\|c\|}  \leqslant \frac{\left\langle c,z-y\right\rangle}{\|c\|} \\
		&& \Leftrightarrow \left\langle c,z-x\right\rangle \leqslant \left\langle c,z-y\right\rangle \\
		&& \Leftrightarrow \left\langle c,z\right\rangle+\left\langle c,-x\right\rangle \leqslant \left\langle c,z\right\rangle+\left\langle c,-y\right\rangle \\
		&& \Leftrightarrow \left\langle c,-x\right\rangle \leqslant \left\langle c,-y\right\rangle \\ 
		&& \Leftrightarrow \left\langle c,x\right\rangle \geqslant \left\langle c,y\right\rangle.
	\end{eqnarray*}  \qed
\end{proof}
Proposition~\ref{prp:coloring} says that problem~\eqref{eq:LP-problem} is equivalent to the following problem:
\begin{equation}\label{eq:coloring_problem}
	\bar x = \arg\min \left\{ \rho_c(x) \middle| x \in M  \right\}.
\end{equation}
\begin{definition}\label{def:objective_projection_on_halfspace}
	Let the half-space $H_i^+$ be recessive with respect to the vector ~$c$. The objective projection $\gamma_i(x)$ of a point $x\in\mathbb{R}^n$ onto the recessive half-space $H_i^+$ is a point defined by the equation
	\begin{equation}\label{eq:objective_projection_on_halfspace}
	\gamma_i(x)= x-\sigma_i(x)c,
	\end{equation}
	where
	\begin{equation*}%\label{key}
		\sigma_i(x)= \min\left\lbrace \sigma \in\mathbb{R}_{\geqslant 0}\; \middle| \; x-\sigma c \in H_i^+ \right\rbrace.
	\end{equation*}
\end{definition}
Examples of objective projections in $\mathbb{R}^2$ are shown in Figure~\ref{Fig01_Objective_projection_to_H_i}.
\begin{figure}[h]
	\centering
	\includegraphics[height=4.5 cm]{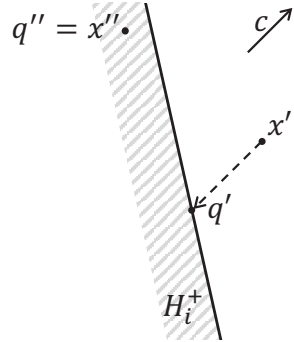}
	\caption{Objective projections in space $\mathbb{R}^2$: $\gamma_i(x')=q'$; $\gamma_i(x'')=q''=x''$.}
	\label{Fig01_Objective_projection_to_H_i}
\end{figure}

\noindent The following Proposition~\ref{prp:objective_projection_on_halfspace} provides an equation for calculating the objective projection onto a half-space that is recessive with respect to the vector $c$.
\begin{proposition}\label{prp:objective_projection_on_halfspace} 
	Let half-space $H_i^+$ defined by the inequality
	\begin{equation}\label{eq:H_i^+}
		\left\langle \tilde a_i,x\right\rangle \leqslant b_i
	\end{equation}
	be recessive with respect to the vector $c$. Let
	\begin{equation}\label{eq:g=pi_c(u)}
		g \notin H_i^+.
	\end{equation}
	Then,
	\begin{equation}\label{eq:gamma_i}
		\gamma_i(g)=g-\frac{\left\langle \tilde a_i,g\right\rangle-b_i}{\left\langle \tilde a_i,c \right\rangle }c.
	\end{equation}
\end{proposition}
\begin{proof}
	According to definition~\ref{def:objective_projection_on_halfspace}, we have
	\begin{equation*}%\label{eq:objective_projection_w}
		\gamma_i(g)=g-\sigma_i(g)c,
	\end{equation*}
	where
	\begin{equation*}%\label{key}
		\sigma_i(x)= \min\left\lbrace \sigma \in\mathbb{R}_{\geqslant 0}\; \middle| \; x-\sigma c \in H_i^+ \right\rbrace.
	\end{equation*}	
	Thus, we need to prove that 
	\begin{equation}\label{eq:lambda}
		\frac{\left\langle \tilde a_i,g\right\rangle-b_i}{\left\langle \tilde a_i,c \right\rangle}=\min\left\lbrace \sigma \in\mathbb{R}_{\geqslant 0}\; \middle| \; x-\sigma c \in H_i^+ \right\rbrace.
	\end{equation}
	Consider the strait line $L$ defined by the parametric equation
	\begin{equation*}%\label{eq:L=g+tau_c}
		L=\left\lbrace g+\tau c\middle| \tau\in \mathbb{R} \right\rbrace.
	\end{equation*}
	Let the point $q$ be the intersection of the line $L$ with the hyperplane $H_i$:
	\begin{equation}\label{eq:q_in_L_cap_H_i}
		q=L\cap H_i.
	\end{equation}
	Then, $q$ must satisfy the equation
	\begin{equation}\label{eq:q=g+tau'c}
		q=g+\tau'c
	\end{equation}
	for some $\tau'\in\mathbb{R}$. Substitute the right side of equation~\eqref{eq:q=g+tau'c} into equation~\eqref{eq:H_i} instead of~$x$:
	\begin{equation*}%\label{eq:}
		\left\langle \tilde a_i,g+\tau'c\right\rangle = b_i.
	\end{equation*}
	It follows that
	\begin{eqnarray}
		&&\left\langle \tilde a_i,g\right\rangle + \tau'\left\langle \tilde a_i,c\right\rangle = b_i,\nonumber\\
		&&\tau'= \frac{b_i-\left\langle \tilde a_i,g\right\rangle}{\left\langle \tilde a_i,c\right\rangle}.\label{eq:tau'}
	\end{eqnarray}
	Substituting the right side of equation~\eqref{eq:tau'} into equation~\eqref{eq:q=g+tau'c} instead of~$\tau'$, we obtain
		\begin{equation*}%^\label{key}
		q=g+\frac{b_i-\left\langle \tilde a_i,g\right\rangle }{\left\langle \tilde a_i,c \right\rangle }c,
	\end{equation*}
	which is equivalent to 
	\begin{equation}\label{eq:q}
		q=g-\frac{\left\langle \tilde a_i,g\right\rangle-b_i}{\left\langle \tilde a_i,c \right\rangle }c.
	\end{equation}
	Since, according to~\eqref{eq:q_in_L_cap_H_i}, $q\in H_i$, equation~\eqref{eq:lambda} will hold if
	\begin{equation}\label{eq:sigma}
		\forall\sigma \in \mathbb{R}_{>0} : \sigma < \frac{\left\langle \tilde a_i,g\right\rangle-b_i}{\left\langle \tilde a_i,c \right\rangle } \Rightarrow g-\sigma c \notin H_i^+
	\end{equation}
	holds. Assume the opposite, that is, there exist $\sigma'>0$ such that
		\begin{equation}\label{eq:sigma'<}
		\sigma' < \frac{\left\langle \tilde a_i,g\right\rangle-b_i}{\left\langle \tilde a_i,c \right\rangle }
	\end{equation}
	and
	\begin{equation}\label{eq:g}
		g-\sigma' c \in H_i^+.
	\end{equation}
	Then, it follows from~\eqref{eq:H_i^+} and~\eqref{eq:g} that
	\begin{equation*}
		\left\langle \tilde a_i,g-\sigma'c \right\rangle \leqslant b_i. 	
	\end{equation*}
	This is equivalent to
		\begin{equation}\label{eq:33}
		\left\langle \tilde a_i,g \right\rangle - b_i \leqslant \sigma' \left\langle \tilde a_i,c \right\rangle.
	\end{equation}
	Proposition~\ref{prp:recessive_halfspace_condition} implies that $\left\langle \tilde a_i,c \right\rangle>0$. Hence, equation~\eqref{eq:33} is equivalent to
		\begin{equation*}%\label{key}
		\sigma'\geqslant\frac{\left\langle \tilde a_i,g\right\rangle-b_i}{\left\langle \tilde a_i,c \right\rangle }.
	\end{equation*}
	Thus, we have a contradiction with~\eqref{eq:sigma'<}. \qed
\end{proof}
\begin{definition}\label{def:objective_projection_on_M}
	Let $g\in H_c$. The \emph{objective projection}~$\gamma_M(g)$ of the point~$g$ onto the polytope~$M$ is the point defined by the following equation:
	\begin{equation}\label{eq:objective_projection_on_halfspace}
		\gamma_M(g)=g-\sigma_M(g)c,
	\end{equation}
	where
	\begin{equation*}%\label{key}
	\sigma_M(g)=\min\left\lbrace \sigma\in\mathbb{R}_{\geqslant 0}\middle| g-\sigma c \in M \right\rbrace.
	\end{equation*}
	If
	\begin{equation*}%\label{key}
		\neg\exists\; \sigma\in\mathbb{R}_{\geqslant 0}:g-\sigma c \in M,
	\end{equation*}
	then we set~$\gamma_M(g)=\vec \infty$, where $\vec\infty$ stands for a point that is infinitely far from the polytope~$M$.
\end{definition}
Examples of objective projections onto polytope~$M$ in~$\mathbb{R}^2$ are shown in Figure~\ref{Fig02_Objective_projection_to_M}.
\begin{figure}[h]
	\centering
	\includegraphics[scale=0.8]{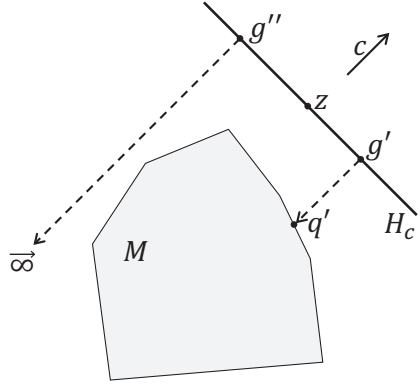}
	\caption{Objective projections onto polytope~$M$ in~$\mathbb{R}^2$: $\gamma_M(g')=q'$; $\gamma_M(g'')=\vec\infty$.}
	\label{Fig02_Objective_projection_to_M}
\end{figure}
\begin{definition}\label{def:receptive_field}
	A receptive field~$\mathfrak{G}(z,\eta,\delta)\subset H_c$ of density~$\delta\in\mathbb{R}_{>0}$ with center~$z\in H_c$ and rank~$\eta\in\mathbb{N}$ is a finite ordered set of points satisfying the following conditions:
		\begin{eqnarray}
		&&z\in \mathfrak{G}(z,\eta,\delta);\label{eq:z_in_G}\\
		&&\forall g\in \mathfrak{G}(z,\eta,\delta) : \|g-z\|\leqslant \eta\delta\sqrt{n};\label{eq:||g-z||<=eta*delta*sqrt(n)}\\
		&&\forall g',g''\in \mathfrak{G}(z,\eta,\delta) : g'\ne g'' \Rightarrow \|g'-g''\|\geqslant \delta;\label{eq:||g'-g''||<=delta}\\
		&&\forall g'\in \mathfrak{G}(z,\eta,\delta)\;\exists g''\in \mathfrak{G}(z,\eta,\delta):\|g'-g''\|=\delta ;\label{eq:||g'-g''||=delta}\\
		&&\forall x\in \operatorname{Co}(\mathfrak{G}(z,\eta,\delta))\;\exists g\in \mathfrak{G}(z,\eta,\delta):\|g-x\|\leqslant \tfrac{1}{2} \delta\sqrt{n}.\label{eq:||g-x||<=delta*sqrt(n)/2}
	\end{eqnarray}
	The points of the receptive field will be called receptive points.
\end{definition}
Here, $\operatorname{Co}(X)$ stands for the convex hull of a finite point set~\mbox{$X=\left\lbrace x^{(1)},\ldots,x^{(K)}\right\rbrace \subset\mathbb{R}^n$}:
\begin{equation*}%\label{key}
	{Co}(X)=\left\lbrace \sum_{i=1}^{K}\lambda_i x^{(i)}\middle| \lambda_i\in\mathbb{R}_{\geqslant 0},\sum_{i=1}^{K}\lambda_i=1\right\rbrace. 			
\end{equation*}
In definition~\ref{def:receptive_field}, condition~\eqref{eq:z_in_G} means that the center of the receptive field belongs to this field. Condition~\eqref{eq:||g-z||<=eta*delta*sqrt(n)} implies that the distance from the central point~$z$ to each point~$g$ of the receptive field does not exceed~$\eta\delta\sqrt{n}$. According to~\eqref{eq:||g'-g''||<=delta}, for any two different points~$g'\ne g''$ of the receptive field, the distance between them cannot be less than~$\delta$. Condition~~\eqref{eq:||g'-g''||=delta} says that for any point~$g'$ of a receptive field, there is a point~$g''$ in this field such that the distance between~$g'$ and $g''$ is equal to~$\delta$. Condition~\eqref{eq:||g-x||<=delta*sqrt(n)/2} implies that for any point~$x$ belonging to the convex hull of the receptive field, there is a point~$g$ in this field such that the distance between~$x$ and~$g$ does not exceed $\tfrac{1}{2} \delta\sqrt{n}$. An example of a receptive field in the space~$\mathbb{R}^3$ is presented in Figure~\ref{Fig03_Rceptive_field}.
\begin{figure}[t]
	\centering
	\includegraphics[scale=0.85]{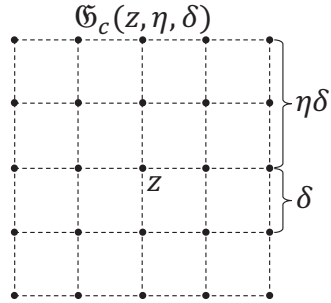}
	\caption{Receptive field in space~$\mathbb{R}^3$.}
	\label{Fig03_Rceptive_field}
\end{figure}

Let us describe a constructive method for building a receptive field. Without loss of generality, we assume that~$c_n\ne 0$. Consider the following set of vectors:
\begin{equation*}
	\begin{gathered}
		{c^{(0)}} = c = ({c_1},{c_2},{c_3},{c_4}, \ldots ,{c_{n - 1}},{c_n}); \hfill \\
		{c^{(1)}} = \left\{ \begin{gathered}
			\left( { - \tfrac{1}{{{c_1}}}\sum\nolimits_{i = 2}^n {c_i^2} ,{c_2},{c_3},{c_4}, \ldots ,{c_{n - 1}},{c_n}} \right),\; \text{if} \;{c_1} \ne 0; \hfill \\
			(1,0, \ldots ,0),\; \text{if} \;{c_1} = 0; \hfill \\
		\end{gathered}  \right. \hfill \\
		{c^{(2)}} = \left\{ \begin{gathered}
			\left( {0, - \tfrac{1}{{{c_2}}}\sum\nolimits_{i = 3}^n {c_i^2} ,{c_3},{c_4}, \ldots ,{c_{n - 1}},{c_n}} \right),\; \text{if} \;{c_2} \ne 0; \hfill \\
			(0,1,0, \ldots ,0),\; \text{if} \;{c_2} = 0; \hfill \\
		\end{gathered}  \right. \hfill \\
		{c^{(3)}} = \left\{ \begin{gathered}
			\left( {0,0, - \tfrac{1}{{{c_3}}}\sum\nolimits_{i = 4}^n {c_i^2} ,{c_4}, \ldots ,{c_{n - 1}},{c_n}} \right),\; \text{if} \;{c_3} \ne 0; \hfill \\
			(0,0,1,0, \ldots ,0),\; \text{if} \;{c_3} = 0; \hfill \\
		\end{gathered}  \right. \hfill \\
		\ldots  \ldots  \ldots  \ldots  \ldots  \ldots  \ldots  \ldots  \ldots  \ldots  \ldots  \ldots  \ldots  \ldots  \ldots  \ldots  \ldots  \ldots  \ldots  \hfill \\
		{c^{(n - 2)}} = \left\{ \begin{gathered}
			\left( {0, \ldots ,0, - \tfrac{1}{{{c_{n - 2}}}}\sum\nolimits_{i = n - 1}^n {c_i^2} ,{c_{n - 1}},{c_n}} \right),\; \text{if} \;{c_{n - 2}} \ne 0; \hfill \\
			(0, \ldots ,0,1,0,0),\; \text{if} \;{c_{n - 2}} = 0; \hfill \\
		\end{gathered}  \right. \hfill \\
		{c^{(n - 1)}} = \left\{ \begin{gathered}
			\left( {0, \ldots ,0, - \tfrac{{c_n^2}}{{{c_{n - 1}}}},{c_n}} \right),\; \text{if} \;{c_{n - 1}} \ne 0; \hfill \\
			(0, \ldots ,0,0,1,0),\; \text{if} \;{c_{n - 1}} = 0. \hfill \\
		\end{gathered}  \right. \hfill \\
	\end{gathered}
\end{equation*}
It is easy to see that
\begin{equation*}%\label{eq:<c_i,c_j>=0}
	\forall i,j\in\{0,1,\ldots,n-1\}, i\ne j:\left\langle c^{(i)},c^{(j)}\right\rangle =0.
\end{equation*}
This means that $c_0,\ldots,c_{n-1}$ is an orthogonal basis in~$\mathbb{R}^n$. In particular,
\begin{equation}\label{eq:<c,c^i>=0}
	\forall i=1,\ldots,n-1:\left\langle c,c^{(i)}\right\rangle =0.
\end{equation}
The following Proposition~\ref{prp:S_c+z=H_} shows that the linear subspace of dimension~\mbox{$(n-1)$} generated by the orthogonal vectors $c_1,\ldots,c_{n-1}$ is a hyperplane parallel to the hyperplane~$H_c$.
\begin{proposition}\label{prp:S_c+z=H_c}
	Define the following linear subspace~$S_c$ of dimension~\mbox{$(n-1)$} in~$\mathbb{R}^n$:
	\begin{equation}\label{eq:S_c}
		S_c=\left\lbrace \sum_{i=1}^{n-1}\lambda_i c^{(i)}\middle|\lambda_i\in\mathbb{R} \right\rbrace.
	\end{equation}
	Then,
	\begin{equation}\label{eq:s+z_in_Hc}
		\forall s\in S_c:s+z\in H_c.
	\end{equation}
\end{proposition}
\begin{proof}
	Let $s\in S_c$, that is
	\begin{equation*}%\label{key}
		s=\lambda_1 c^{(1)}+\ldots+\lambda_{n-1} c^{(n-1)}.
	\end{equation*}
	Then,
	\begin{equation*}
		\left\langle c,(s+z)-z\right\rangle =\lambda_1\left\langle c,c^{(1)}\right\rangle +\ldots+\lambda_{n-1}\left\langle c,c^{(n-1)}\right\rangle.
	\end{equation*}
	In view of~\eqref{eq:<c,c^i>=0}, this implies
	\begin{equation*}
		\left\langle c,(s+z)-z\right\rangle =0.
	\end{equation*}
	Comparing this with~\eqref{eq:H_c}, we obtain~$s+z\in H_c$.\qed
\end{proof}	
Define the following set of vectors:
\begin{equation}\label{eq:||e_i||=c_i/||c_i||}
	e^{(i)}=\frac{c^{(i)}}{\|c^{(i)}\|}\;(i=1,\ldots,n-1).
\end{equation}
It is easy to see that the set $\{e_1,\ldots,e_{n-1}\}$ is an orthonormal basis of subspace~$S_c$.

The procedure for constructing a receptive field is presented as Algorithm~\ref{alg:receptive_field}. This algorithm constructs a receptive field~$\mathfrak{G}(z,\eta,\delta)$ consisting of
\begin{equation}\label{eq:|G|}
	K_{\mathfrak{G}}=(2\eta+1)^{n-1}
\end{equation}
points. These points are arranged at the nodes of a regular lattice having the form of a hypersquare (a hypercube of dimension $n-1$) with the edge length equal to $2\eta\delta$. The edge length of the unit cell is~$\delta$. According to Step~13 of Algorithm~\ref{alg:receptive_field} and Proposition~\ref{prp:S_c+z=H_c}, this hypersquare lies in the hyperplane $H_c$ and has the center at the point $z$.
\begin{figure}[t]\begin{center}\begin{minipage}{0.6\textwidth}
			\begin{algorithm}[H]\caption{Building receptive field  $\mathfrak{G}(z,\eta,\delta)$}\label{alg:receptive_field}
				\begin{algorithmic}[1]
					\Require $z\in H_c$, $\eta\in\mathbb{N}$, $\delta\in\mathbb{R}_{>0}$
					\State $\mathfrak{G}\ceq\emptyset$
					\For{$i_{n-1}=0\ldots2\eta$}
					\State $s_{n-1}\ceq i_{n-1}\delta-\eta\delta$
					\For{$i_{n-2}=0\ldots2\eta$}
					\State $s_{n-2}\ceq i_{n-2}\delta-\eta\delta$
					\State \ldots
					\For{$i_1=0\ldots2\eta$}
					\State $s_1\ceq i_1\delta-\eta\delta$
					\State $s\ceq \mathbf{0}$
					\For{$j=1\ldots n-1$}
					\State $s\ceq s+s_je^{(j)}$
					\EndFor
					\State $\mathfrak{G}\ceq\mathfrak{G}\cup \{s+z\}$
					\EndFor
					\EndFor
					\EndFor
\end{algorithmic}\end{algorithm}\end{minipage}\end{center}\end{figure}
The drawback of Algorithm~\ref{alg:receptive_field} is that the number of nested  \textbf{for} loops depends on the dimension of the space. This issue can be solved using the function $\operatorname{G}$, which calculates a point of the receptive field by its ordinal number (numbering starts from zero; the order is determined by Algorithm~\ref{alg:receptive_field}). The implementation of the function $\operatorname{G}$ is represented as Algorithm~\ref{alg:function_G}.
\begin{figure}[t]\begin{center}\begin{minipage}{0.9\textwidth}
			\begin{algorithm}[H]\caption{Function~$\operatorname{G}$ calculates a receptive point by its number~$k$}\label{alg:function_G}
				\begin{algorithmic}[1]
					\Require $z\in H_c$, $\eta\in\mathbb{N}$, $\delta\in\mathbb{R}_{>0}$
					\Function {G}{$k,n,z,\eta,\delta$}
					\For{$j=(n-1)\ldots1$}
					\State $i_j\ceq \left\lfloor k/(2\eta+1)^{j-1} \right\rfloor$
					\State $k\ceq k \mod (2\eta+1)^{j-1}$
					\EndFor
					\State $g\ceq z$
					\For{$j=1\ldots(n-1)$}
					\State $g\ceq g+(i_j\delta-\eta\delta)e^{(j)}$
					\EndFor
					\State $\operatorname{G}\ceq g$
					\EndFunction
\end{algorithmic}\end{algorithm}\end{minipage}\end{center}\end{figure}
The following Proposition~\ref{prp:time_complexity_of_G} provides an estimation of time complexity of Algorithm~\ref{alg:function_G}.
\begin{proposition}\label{prp:time_complexity_of_G}
	Algorithm~\ref{alg:function_G} allows an implementation that has time complexity\footnote{Here, time complexity refers to the number of arithmetic and comparison operations required to execute the algorithm.}
	\begin{equation}\label{eq:c_G}
		c_G=4n^2+5n-9,
	\end{equation}
	where $n$ is the space dimension.
\end{proposition}
\begin{proof}
	Consider Algorithm ~\ref{alg:G_implementation} representing a low-level implementation of Algorithm~\ref{alg:function_G}.
	\begin{figure}[t]\begin{center}\begin{minipage}{0.7\textwidth}
			\begin{algorithm}[H]\caption{Low-level implementation of Algorithm~\ref{alg:function_G} }\label{alg:G_implementation}
				\begin{algorithmic}[1]
					\State $p \ceq 2\eta+1$; $r \ceq \eta\delta$; $h \ceq p^{n-2}$; $g \ceq z$
					\State $j \ceq n-1$
					\Repeat
					\State $l_j\ceq \left\lfloor k/h \right\rfloor$
					\State $k\ceq k \mod h$
					\State $h\ceq h/p$
					\State $j \ceq j-1$
					\Until{$j=0$}
					\State $j \ceq 1$
					\Repeat
					\State $w_j \ceq l_j\delta-r$
					\State $i \ceq 1$
					\Repeat
					\State $g_i\ceq g_i+w_je_i^{(j)}$
					\State $i\ceq i+1$
					\Until{$i>n$}
					\State $j \ceq j+1$
					\Until{$j=n$}
	\end{algorithmic}\end{algorithm}\end{minipage}\end{center}\end{figure}
	Values calculated in Steps~1--2 of the algorithm~\ref{alg:G_implementation} do not depend on the receptive point number~$k$ and therefore can be considered constants. In Steps~3--8, the \textbf{repeat/until} loop runs $(n-1)$ times and requires $c_{3:8}=5(n-1)$ operations. In steps~13--16, the nested \textbf{repeat/until} loop runs $n$ times and requires $c_{13:16}=4n$ operations. In steps~10--18, the external \textbf{repeat/until} loop runs $(n-1)$ times and requires $c_{10:18}=(4+c_{13-16})(n-1)=4(n^2-1)$ operations. In total, we obtain
	\begin{equation*}%\label{eq:c_G}
		c_G=c_{3:8}+c_{10:18}=4n^2+5n-9.
	\end{equation*}
	\qed
\end{proof}

\begin{corollary}
	Time complexity of Algorithm~\ref{alg:function_G} can be estimated as $O(n^2)$.
\end{corollary}

\begin{definition}
	Let $z\in H_c$. Fix $\eta\in\mathbb{N}$, $\delta\in\mathbb{R}_{>0}$. The image $\mathfrak{I}(z,\eta,\delta)$ generated by the receptive field $\mathfrak{G}(z,\eta,\delta)$ is an ordered set of real numbers defined by the equation
	\begin{equation}\label{eq:image}
		\mathfrak{I}(z,\eta,\delta)=\left\lbrace \rho_c(\gamma_M(g)) \middle| g\in \mathfrak{G}(z,\eta,\delta) \right\rbrace.
	\end{equation}
	The order of the real numbers in the image is determined by the order of the respective receptive points.
\end{definition}
The following Algorithm~\ref{alg:image} implements the function $\mathfrak{I}(z,\eta,\delta)$ building an image as a list of real numbers.
\begin{figure}[H]\begin{center}\begin{minipage}{0.5\textwidth}
			\begin{algorithm}[H]\caption{Building image $\mathfrak{I}(z,\eta,\delta)$}\label{alg:image}
				\begin{algorithmic}[1]
					\Require $z\in H_c$, $\eta\in\mathbb{N}$, $\delta\in\mathbb{R}_{>0}$
					\Function {$\mathfrak{I}$}{$z,\eta,\delta$}
					\State $\mathfrak{I}\ceq [\,]$
					\For{$k=0\ldots((2\eta+1)^{n-1}-1)$}
					\State $g_k\ceq \operatorname{G}(k,n,z,\eta,\delta)$
					\State $\mathfrak{I}\ceq\mathfrak{I} \doubleplus [ \rho_c(\gamma_M(g_k))]$
					\EndFor
					\EndFunction
\end{algorithmic}\end{algorithm}\end{minipage}\end{center}\end{figure}
\noindent Here, $[\,]$ stands for the empty list, and~$\doubleplus$ stands for the operation of list concatenation.

Let $\left\langle \tilde a_i,c\right\rangle >0$. This means that the half-space~$H_i^+$ is recessive with respect to the vector $c$ (see Proposition~\ref{prp:recessive_halfspace_condition}). Let there be a point~$u\in H_i\cap M$. Assume that we managed to create an artificial neural network DNN, which receives the image~$\mathfrak{I}(\pi_c(u),\eta,\delta)$ as an input, and outputs the point~$u'$ such that
\begin{equation*}%\label{key}
	u'=\arg \min \left\lbrace \rho_c(x)\middle| x\in H_i\cap M\right\rbrace.
\end{equation*}
Then, we can build the following Algorithm~\ref{alg:DNN_movement_method} solving the linear programming problem~\eqref{eq:coloring_problem} using DNN.
\begin{figure}[H]
	\begin{center}
		\begin{minipage}{0.75\textwidth}
			\begin{algorithm}[H]
				\caption{Linear programming using DNN}\label{alg:DNN_movement_method}
				\begin{algorithmic}[1]
					\Require $u^{(1)}\in H_i\cap M,\; \left\langle \tilde a_i,c\right\rangle >0, \; z\in H_c;\; \eta \in \mathbb{N},\; \delta \in \mathbb{R}_{>0}$
					\State $k \ceq 1$
					\Repeat
					\State $\mathcal{I} \ceq \mathfrak{I}(u^{(k)},\eta,\delta)$
					\State $u^{(k+1)}\ceq \operatorname{DNN}(\mathcal{I})$
					\State $k \ceq k+1$
					\Until{$u^{(k)}\ne u^{(k-1)}$}
					\State $\bar x\ceq u^{(k)}$
					\State \textbf{stop}
				\end{algorithmic}
			\end{algorithm}
		\end{minipage}
	\end{center}
\end{figure}
\noindent Only an outline of the forthcoming algorithm is presented here, which needs further formalization, detalization and refinement.
\section{Parallel Algorithm for Building LP Problem Image}\label{sec:parallel_algorithm}
When solving LP problems of large dimension with a large number of constraints, Algorithm ~\ref{alg:image} of building the LP problem image can incur significant runtime overhead. This section presents a parallel version of Algorithm~\ref{alg:image}, which significantly reduces the runtime overhead of building the image of a large-scale LP problem. The parallel implementation of Algorithm~\ref{alg:image} is based on the BSF parallel computation model~\cite{Sokolinsky2021,Sokolinsky2018}. The BSF model is intended for a cluster computing system, uses the master/worker paradigm and requires the representation of the algorithm in the form of operations on lists using higher-order functions \emph{Map} and \emph{Reduce} defined in the Bird--Meertens formalism~\cite{Bird1988}. The BSF model also provides a cost metric for the analytical evaluation of the scalability of a parallel algorithm that meets the specified requirements. Examples of the BSF model application can be found in~\cite{Ezhova2018,Sokolinsky2020b,Sokolinsky2020a,Sokolinsky2021b,Sokolinsky2021d}.

Let us represent Algorithm~\ref{alg:image} in the form of operations on lists using higher-order functions \emph{Map} and \emph{Reduce}. We use the list of ordinal numbers of inequalities of system~\eqref{eq:inequalities} as a list, which is the second parameter of the higher-order function \emph{Map}:
\begin{equation}\label{eq:map-list}
	\mathcal{L}_{map}=\left[1,\ldots, m\right].
\end{equation}
Designate $\mathbb{R}_\infty=\mathbb{R}\cup \{\infty\}$. We define a parameterized function \[\operatorname{F}_k:\left\lbrace 1,\ldots,m\right\rbrace \to \mathbb{R}_{\infty},\]
which is the first parameter of a higher-order function \emph{Map}, as follows:
\begin{equation}\label{eq:F_k}
	\operatorname{F}_k(i) = \left\{ \begin{gathered}
		\rho_c\left( \gamma_i(g_k)\right),\; \text{if} \; \left\langle \tilde a_i,c\right\rangle > 0 \; \text{and} \; \gamma_i(g_k) \in M; \hfill \\
		\infty,\; \text{if} \; \left\langle \tilde a_i,c\right\rangle \leqslant 0 \; \text{or} \; \gamma_i(g_k) \notin M. \hfill
	\end{gathered}  \right.
\end{equation}
where $g_k=G(k,n,z,\eta,\delta)$ (see Algorithm~\ref{alg:function_G}), and $\gamma_i(g_k)$ is calculated by equation~\eqref{eq:gamma_i}. Informally, the function~$\operatorname{F}_k$ maps the ordinal number of half-space $H_i^+$ to the distance from the objective projection to the objective hyperplane if $H_i^+$ is recessive with respect to $c$ (see Proposition~\ref{prp:recessive_halfspace_condition}) and the objective projection belongs to $M$. Otherwise, $\operatorname{F}_k$ returns the special value $\infty$.

The higher-order function \emph{Map} transforms the list $\mathcal{L}_{map}$ into the list $\mathcal{L}_{reduce}$ by applying the function $\operatorname{F}_k$ to each element of the list $\mathcal{L}_{map}$:
\begin{equation*}%\label{eq:reduce_list}
	\mathcal{L}_{reduce}=\operatorname{Map}\left( \operatorname{F}_k, \mathcal{L}_{map}\right)  = \left[\operatorname{F}_k(1),\ldots, \operatorname{F}_k(m)\right] = \left[\rho_1,\ldots, \rho_m\right].
\end{equation*}
Define the associative binary operation $\reduceop\;: \mathbb{R}_{\infty}\to\mathbb{R}_{\infty}$ as follows:
\begin{eqnarray*}%\label{key}
	\infty \reduceop\; \infty &=& \infty;\\
	\forall \alpha \in \mathbb{R} : \alpha \reduceop\; \infty &=& \alpha;\\
	\forall \alpha,\beta \in \mathbb{R} : \alpha \reduceop\; \beta &=& \min(\alpha,\beta).\\
\end{eqnarray*}
Informally, the operation~$\reduceop\;$ calculates the minimum of two numbers.

The higher-order function \emph{Reduce} folds the list $\mathcal{L}_{reduce}$ to the single value \mbox{$\rho \in \mathbb{R}_\infty$} by sequentially applying the operation $\reduceop$\, to the entire list:
\begin{equation*}%\label{key}
	\operatorname{Reduce}(\reduceop\;,\mathcal{L}_{reduce}) = \rho_1 \reduceop\; \rho_2  \reduceop\; \ldots  \reduceop\; \rho_m = \rho.
\end{equation*}

The Algorithm~\ref{alg:map_reduce} builds the image~$\mathfrak{I}$ of LP problem using higher-order functions $\emph{Map}$ and $\emph{Reduce}$.
\begin{figure}[t]\begin{center}\begin{minipage}{0.75\textwidth}
			\begin{algorithm}[H]\caption{Building the image~$\mathfrak{I}$ by $\emph{Map}$ and $\emph{Reduce}$ }\label{alg:map_reduce}
				\begin{algorithmic}[1]
					\Require $z\in H_c$, $\eta\in\mathbb{N}$, $\delta\in\mathbb{R}_{>0}$
					\State \textbf{input} $n,m,A,b,c,z,\eta,\delta$
					\State $\mathfrak{I}\ceq [\;]$
					\State $\mathcal{L}_{map} \ceq [1,\ldots,m]$
					\For{$k=0\ldots((2\eta+1)^{n-1}-1)$}
					\State $\mathcal{L}_{reduce} \ceq \operatorname{Map} (\operatorname{F}_k, \mathcal{L}_{map}) $
					\State $\rho \ceq \operatorname{Reduce} (\reduceop\;, \mathcal{L}_{reduce})$
					\State $\mathfrak{I}\ceq\mathfrak{I} \doubleplus [\rho]$
					\EndFor
					\State \textbf{output} $\mathfrak{I}$					\State \textbf{stop}
\end{algorithmic}\end{algorithm}\end{minipage}\end{center}\end{figure}
The parallel version of Algorithm~\ref{alg:map_reduce} is based on algorithmic template~2 in~\cite{Sokolinsky2021}. The result is presented as Algorithm~\ref{alg:parallel}.
\begin{algorithm}[H]
	\caption{Parallel algorithm of building image $\mathfrak{I}$}\label{alg:parallel}
 	\begin{multicols}{2}
	\begin{center}
		\textbf{Master} \\
		\textbf{$l$th Worker (\emph{l}=0,\dots,\emph{L-1})}
	\end{center}
\end{multicols}
\algrule
\begin{multicols}{2}
	\begin{algorithmic}[1]
		\State \textbf{input} $n$
		\State $\mathfrak{I}\ceq [\;]$
		\State $k \ceq 0$
		\Repeat
		\State \textbf{SendToWorkers} $k$
		\State
		\State
		\State \textbf{RecvFromWorkers} $[\rho_0,\ldots,\rho_{L-1}]$
		\State $\rho \ceq \operatorname{Reduce} \left( \reduceop\;, [\rho_0,\ldots,\rho_{L-1}]\right) $
		\State $\mathfrak{I}\ceq\mathfrak{I} \doubleplus [\rho]$
		\State $k\ceq k+1$
		\State $exit \ceq \left(k\geqslant(2\eta+1)^{n-1}\right) $
		\State \textbf{SendToWorkers} $exit$
		\Until {$exit$}
		\State \textbf{output} $\mathfrak{I}$
		\State \textbf{stop}
	\end{algorithmic}
	\begin{algorithmic}[1]%---------------------------------------------------------------
		\State \textbf{input} $n,m,A,b,c,z,\eta,\delta$
		\State $L\ceq \mathrm{NumberOfWorkers}$
		\State $\mathcal{L}_{map(l)} \ceq [lm/L,\ldots,((l+1)m/L)-1]$
		\Repeat
		\State \textbf{RecvFromMaster} $k$
		\State $\mathcal{L}_{reduce(l)} \ceq \operatorname{Map} \left( \operatorname{F}_k, \mathcal{L}_{map(l)}\right)  $
		\State $\rho_l \ceq \operatorname{Reduce} \left( \reduceop\;, \mathcal{L}_{reduce(l)}\right) $
		\State \textbf{SendToMaster} $\rho_l$
		\State
		\State
		\State
		\State
		\State \textbf{RecvFromMaster} $exit$
		\Until {$exit$}
		\State
		\State \textbf{stop}
	\end{algorithmic}
\end{multicols}
\end{algorithm}

Let us explain the steps of Algorithm~\ref{alg:parallel}. For simplicity, we assume that the number of constraints~$m$ is a multiple of the number of workers~$L$. We also assume that the numbering of inequalities starts from zero. The parallel algorithm includes~$L+1$ processes: one master process and~$L$ worker processes. The master manages the computations. In Step 1, the master reads the space dimension~$n$. In Step 2 of the master, the image variable~$\mathfrak{I}$ is initialized to the empty list. Step~3 of the master assigns zero to the iteration counter~$k$. At Steps ~\mbox{4--14}, the master organizes the \textbf{repeat/until} loop, in which the image~$\mathfrak{I}$ of the LP problem is built. In Step~5, the master sends the receptive point number $g_k$ to all workers. In Step~8, the master is waiting the particular results from all workers. These particular results are folded to a single value, which is added to the image $\mathfrak{I}$ (Steps~9--10 of the master). Step~11 of the master increases the iteration counter~$k$ by~1. Step~12 of the master assigns the logical value $\left( k\geqslant(2\eta+1)^{n-1}\right)$ to the Boolean variable \emph{exit}. In Step~13, the master sends the value of the Boolean variable \emph{exit} to all workers. According to~\eqref{eq:|G|}, $exit=false$ means that not all the points of the receptive field have been processed. In this case, the control is passed to the next iteration of the external \textbf{repeat/until} loop (Step~14 of the master). After exiting the \textbf{repeat/until} loop, the master outputs the constructed image $\mathfrak{I}$ (Step~15) and terminates its work (Step~16).

All workers execute the same program codes but with different data. In~Step~3, the $l$th worker defines its own sublist. In Step~4, the worker enters the \textbf{repeat/until} loop. In Step~5, it receives the number~$k$ of the next receptive point. In Step~6, the worker processes its sublist~$\mathcal{L}_{map(l)}$ using the higher-order function \textit{Map}, which applies the parameterized function~$\operatorname{F}_k$, defined by~\eqref{eq:F_k}, to each element of the sublist. The result is the sublist~$\mathcal{L}_{reduce(l)}$, which includes the distances~$\operatorname{F}_k(i)$ from the objective hyperplane~$H_c$ to the objective projections of the receptive point~$g_k$ onto hyperplanes~$H_i$ for all~$i$ from the sublist~$\mathcal{L}_{map(l)}$. In Step~7, the worker uses the higher-order function \textit{Reduce} to fold the sublist $\mathcal{L}_{reduce(l)}$ to the single value of $\rho_l$, using the associative binary operation $\reduceop$\;, which calculates the minimum distance. The computed particular result is sent to the master (Step~8 of the worker). In Step~13, the worker is waiting for the master to send the value of the Boolean variable $exit$. If the received value is false, the worker continues executing the \textbf{repeat/until} loop (Step~14 of the worker). Otherwise, the worker process is terminated in Step~16.

Let us obtain an analytical estimation of the \emph{scalability bound} of the parallel algorithm~\ref{alg:parallel} using the cost metric of the BSF parallel computation model~\cite{Sokolinsky2021}. Here, the scalability bound means the number of workers at which maximum speedup is achieved. The cost metric of the BSF model includes the following cost parameters for the \textbf{repeat/until} loop (Steps 4--14) of the parallel algorithm~\ref{alg:parallel}:
\begin{tabbing}
	MM. \= M \= MMMMMMMMMMMMMMMMMMMMMMMMMMMMMMMMMMMMMMMMMMMMMMMMMMMMMMMMMMMMMMMMMMMMMMMMMMMMMM \kill
	$m$\> : \> length of the list $\mathcal{L}_{map}$;\\
	$D$\> : \> latency (the time taken by the master to send one byte message\\
	\>\> to a single worker);\\
	${t_c}$\> : \> the time taken by the master to send the coordinates of the receptive\\
	\>\> point to a single worker and receive the computed value from it\\
	\>\> (including latency);\\
	${t_{Map}}$\> : \> the time taken by a single worker to process the higher-order function\\ 
	\>\> \textit{Map} for the entire list $\mathcal{L}_{map}$;\\
	${t_a}$\> : \> the time taken by computing the binary operation $\reduceop$\;.
\end{tabbing}
According to equation~(14) from ~\cite{Sokolinsky2021}, the scalability bound of the Algorithm~\ref{alg:parallel} can be estimated as follows:
\begin{equation}\label{eq:scalability}
	{L_{max}} = \frac{1}{2}\sqrt {{{\left( {\frac{{{t_c}}}{{{t_a}\ln 2}}} \right)}^2} + \frac{{{t_{Map}}}}{{{t_a}}} + 4m}  - \frac{{{t_c}}}{{{t_a}\ln 2}}.
\end{equation}
Calculate estimations for the time parameters of equation ~\eqref{eq:scalability}. To do this, we will introduce the following notation for a single iteration of the \textbf{repeat/until} loop (Steps 4--14) of Algorithm~\ref{alg:parallel}:
\begin{tabbing}
	MM. \= M \= MMMMMMMMMMMMMMMMMMMMMMMMMMMMMMMMMMMMMMMMMMMMMMMMMMMMMMMMMMMMMMMMMMMMMMM \kill
	${c_c}$\> : \> the quantity of numbers sent from the master to the worker and\\
	\>\> back within one iteration;\\
	${c_{Map}}$\> : \> the quantity of arithmetic and comparison operations computed in\\
	\>\> Step~5 of the sequential algorithm~\ref{alg:map_reduce};\\
	${c_a}$\> : \> the quantity of arithmetic and comparison operations required\\ 
	\>\> to compute the binary operation $\reduceop$\,.\\
\end{tabbing}
At the beginning of every iteration, the master sends each worker the receptive point number~$k$. In response, the worker sends the distance from the receptive point~$g_k$ to its objective projection. Therefore,
\begin{equation}\label{eq:c_c}
	c_c = 2.
\end{equation}
In the context of Algorithm~\ref{alg:map_reduce}
\begin{equation}\label{eq:c_Map_0}
	c_{Map} = \left( c_G+c_{F_k}\right)m,
\end{equation}
where $c_G$ is the number of operations taken to compute coordinates of the point~$g_k$, and~$c_{F_k}$ is the number of operations required to calculate the value of~$\operatorname{F}_k(i)$, assuming that the coordinates of the point~$g_k$ have already been calculated. The estimation of~$c_G$ is provided by Proposition~\ref{prp:time_complexity_of_G}. Let us estimate~$c_{F_k}$. According to~\eqref{eq:gamma_i}, calculating the objective projection $\gamma_i(g)$ takes \mbox{$(6n-2)$} arithmetic operations. It follows from ~\eqref{eq:distance} that the calculation of $\rho_c(x)$ takes $(5n-1)$ arithmetic operations. Inequalities~\eqref{eq:inequalities} imply that checking the condition~$x\in M$ takes~$m(2n-1)$ arithmetic operations and~$m$ comparison operations. Hence, $\operatorname{F}_k(i)$ takes a total of  $(2mn+11n-3)$ operations. Hence,
\begin{equation}\label{eq:c_F_k}
	c_{F_k} = 2mn+11n-3.
\end{equation}
Substituting the right-hand sides of equations~\eqref{eq:c_G} and~\eqref{eq:c_F_k} in~\eqref{eq:c_Map_0}, we obtain
\begin{equation}\label{eq:c_Map}
	c_{Map} = 4n^2m+2m^2n+16nm-12m.
\end{equation}
To perform the binary operation $\reduceop$\;, one comparison operation must be executed:
\begin{equation}\label{eq:c_a}
	c_a = 1.
\end{equation}
Let $\tau_{op}$ stand for average execution time of arithmetic and comparison operations, and $\tau_{tr}$ stand for average time of sending a single real number (excluding the latency). Then, using equations~\eqref{eq:c_c}, \eqref{eq:c_Map}, and~\eqref{eq:c_a} we obtain
\begin{eqnarray}
	&&t_c=c_c\tau_{tr}+2D=2(\tau_{tr}+D);\label{eq:t_c} \\
	&&t_{Map}=c_{Map}\tau_{op}=(4n^2m+2m^2n+16nm-12m)\tau_{op};\label{eq:t_Map} \\
	&&t_a=c_a\tau_{op}=\tau_{op}.\label{eq:t_a}
\end{eqnarray}
Substituting the right-hand sides of equations~\eqref{eq:t_c} -- \eqref{eq:t_a} in~\eqref{eq:scalability}, we obtain the following estimations of the scalability bound of Algorithm~\ref{alg:parallel}:
\begin{equation*}%\label{key}
	L_{max} = \frac{1}{2}\sqrt {{{\left( {\frac{2(\tau_{tr}+D)}{{{\tau _{op}}\ln 2}}} \right)}^2} + 4n^2m+2m^2n+16nm-12m}  - \frac{2(\tau_{tr}+D)}{{{\tau _{op}}\ln 2}}.
\end{equation*}
where~$n$ is the space dimension, $m$ is the number of constraints, $D$ is the latency. For large values of $m$ and $n$, this is equivalent to
\begin{equation}\label{eq:esimation1}
	L_{max} \approx O(\sqrt{2n^2m+m^2n+8nm-6m}).
\end{equation}
If we assume that $m=O(n)$, then it follows from ~\eqref{eq:esimation1} that
\begin{equation}\label{eq:esimation2}
	L_{max} \approx O(n\sqrt{n}),
\end{equation}
where~$n$ is the space dimension. The estimation~\eqref{eq:esimation2} allows us to conclude that Algorithm~\ref{alg:parallel} scales very well
\footnote{
	Let $L_{max}=O(n^\alpha)$. We say: the algorithm \emph{scales perfectly} if~$\alpha > 1$; the algorithm \emph{scales well} if~$\alpha = 1$; the algorithm demonstrates \emph{limited scalability} if~$0<\alpha < 1$; the algorithm does \emph{not scale} if~$\alpha = 0$.
}.
In the following section, we will verify the analytical estimation~\eqref{eq:esimation2} by conducting large-scale computational experiments on a real cluster computing system.
\section{Computational Experiments}\label{sec:computational_experiments}
We performed a parallel implementation of the algorithm~\ref{alg:parallel} in the form of the ViLiPP (Visualization of Linear Programming Problem) program in C++ using a BSF-skeleton~\cite{Sokolinsky2021a}. The BSF-skeleton based on the BSF parallel computation model encapsulates all aspects related to parallelization of the program using the MPI~\cite{Gropp2012}  library and the OpenMP~\cite{Kale2019} programming interface. The source code of the ViLiPP program is freely available on the Internet at \url{https://github.com/nikolay-olkhovsky/LP-visualization-MPI}. Using the parallel program ViLiPP, we conducted experiments to evaluate the scalability of Algorithm~\ref{alg:parallel} on the cluster computing system ``Tornado SUSU''~\cite{Kostenetskiy2018}, the characteristics of which are presented in Table~\ref{tbl:tornado-susu}.
\begin{table}[t]
	\caption{Specifications of the ``Tornado SUSU'' computing cluster}
	\centering
	\begin{tabular}{l|l}
		\hline
		Parameter & Value \\
		\hline
		Number of processor nodes & 480 \\
		Processor & Intel Xeon X5680 (6 cores, 3.33 GHz) \\
		Processors per node & 2\\
		Memory per node & 24 GB DDR3\\
		Interconnect & InfiniBand QDR (40 Gbit/s) \\
		Operating system & Linux CentOS\\
		\hline
	\end{tabular}\label{tbl:tornado-susu}
\end{table}

To conduct computational experiments, we constructed three random LP problems using the FRaGenLP problem generator~\cite{Sokolinsky2021b}. The parameters of these problems are given in Table~\ref{tbl:problems}. In all cases, the number of non-zero values of the matrix $A$ of problem~\eqref{eq:LP-problem} was~100\%. For all problems, the rank~$\eta$ of the receptive field was assumed to be equal to~$2$. In accordance to equation~\eqref{eq:|G|}, the receptive field cardinality demonstrated an exponential growth with an increase of space dimension.
\begin{table}[t]
	\caption{Parameters of test LP problems}
	\centering
	\begin{tabular}{p{1.2cm}|p{1.2cm}|p{1.6cm}|p{1.7cm}|p{2.2cm}}
		\hline
		Problem ID & Dimen\-sion & Number of constraints & Non-zero values in $A$ & Receptive field cardinality \\
		\hline
		$LP7$ & 7 & 4016 & 100\% & 15\,625 \\
		$LP6$ & 6 & 4014 & 100\% & 3\,125 \\
		$LP5$ & 5 & 4012 & 100\% & 625 \\
		\hline
	\end{tabular}\label{tbl:problems}
\end{table}

The results of computational experiments are presented in Table~\ref{tbl:times} and in Fig.~\ref{Fig04_Speedup}. In all runs,  a separate processor node  was allocated for each worker. One more separate processor node was allocated for the master. Computational experiments shows that the ViLiPP program scalability bound increases with the increase of the problem dimension. For LP5, the maximum of speedup curve is reached around 190 nodes. For LP6, the maximum is located around 260 nodes. For LP7, the scalability bound is approximately equal to 326 nodes. At the same time, there is an exponential increase of the runtime of building the LP problem image. Building the LP5 problem image takes 10 seconds on 11 processor nodes. Building the LP7 problem image takes 5 minutes on the same number of nodes. An additional computational experiment showed that building an image of the problem with $n= 9$ takes 1.5 hours on 11 processor nodes.

\begin{figure}[t]\begin{center}\begin{minipage}{0.75\textwidth}
			\begin{table}[H]
				\caption{Runtime of building LP  problem image (sec.)}
				\centering
				\begin{tabular}{p{2.2cm}|p{1.8cm}|p{1.8cm}|p{1.8cm}}
					\hline
					Number of processor nodes & LP5 & LP6 & LP7 \\
					\hline
					11 	& 9.81 & 54.45 & 303.78 \\
					56 	& 1.93 & 10.02 & 59.43\\
					101 & 1.55 & 6.29 & 33.82 \\
					146 & 1.39 & 4.84 & 24.73 \\
					191 & 1.35 & 4.20 & 21.10 \\
					236 & 1.38 & 3.98 & 19.20 \\
					281 & 1.45 & 3.98 & 18.47 \\
					326 & 1.55 & 4.14 & 18.30 \\
					\hline
				\end{tabular}\label{tbl:times}
			\end{table}
\end{minipage}\end{center}\end{figure}

The conducted experiments show that on the current development level of higgh-performance computing, the proposed method is applicable to solving LP problems that include up to 100 variables and up to 100\,000 constraints.
\begin{figure}[t]
	\centering
	\includegraphics[scale=0.5]{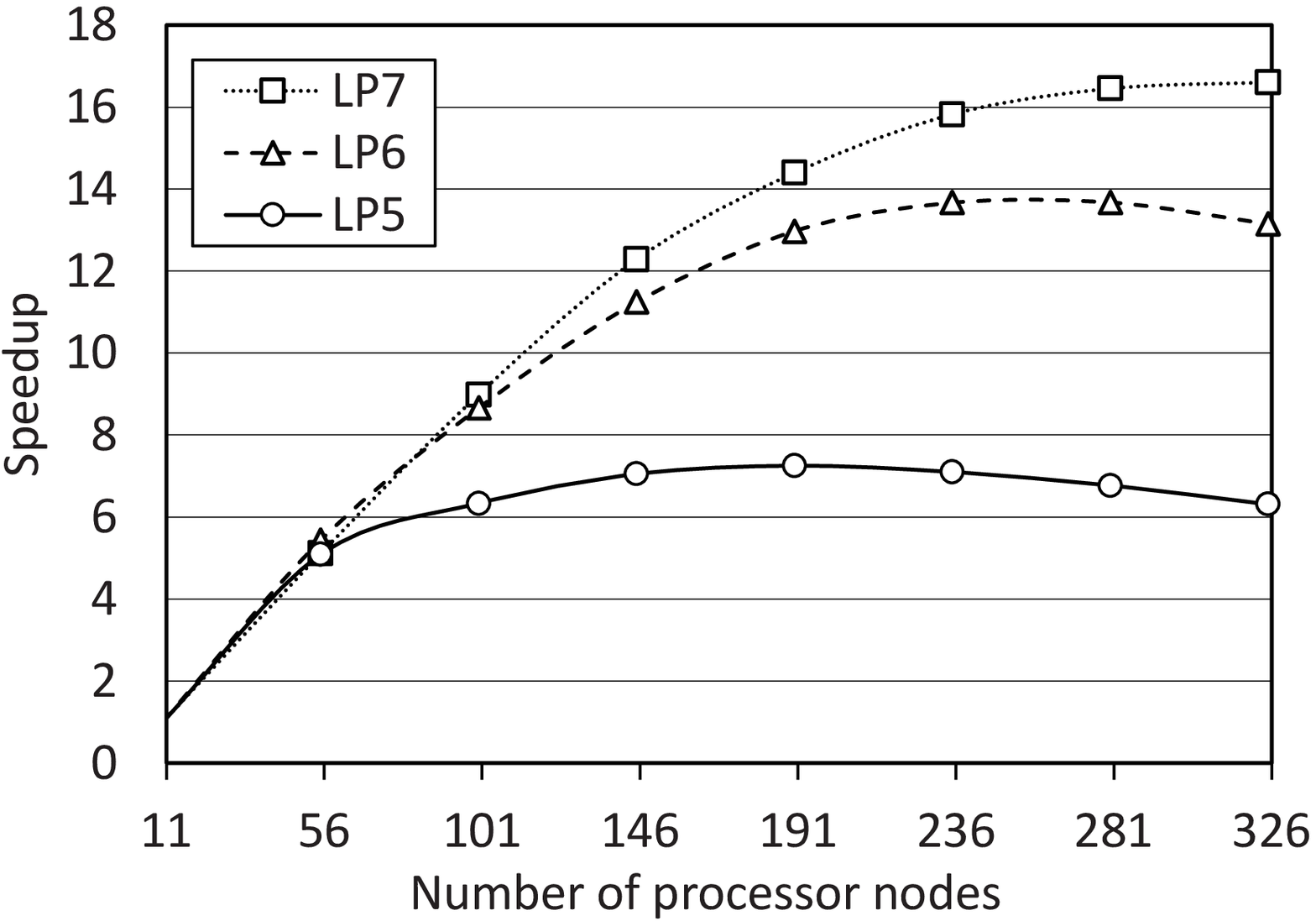}
	\caption{ViLiPP parallel program speedup for LP problems of various sizes.}
	\label{Fig04_Speedup}
\end{figure}
\section*{Conclusion}\label{sec:conclusion}
The main contribution of this work is a mathematical model of the visual representation of a multidimensional linear programming problem of finding the maximum of a linear objective function in a feasible region. The central element of the model is the receptive field, which is a finite set of points located at the nodes of a square lattice constructed inside a hypercube. All points of the receptive field lie in the objective hyperplane orthogonal to the vector $c=(c_1,\ldots,c_n)$, which is composed of the coefficients of the linear objective function. The target hyperplane is placed so that for any point $x$ from the feasible region and any point $z$ of the objective hyperplane, the inequality $\left\langle c,x\right\rangle < \left\langle c,z\right\rangle$ holds. We can say that the receptive field is a multidimensional abstraction of the digital camera image sensor. From each point of the receptive field, we construct a ray parallel to the vector $c$ and directed to the side of the feasible region. The point at which the ray hits the feasible region is called the objective projection. The image of a linear programming problem is a matrix of dimension $(n-1)$, in which each element is the distance from the point of the receptive field to the corresponding point of objective projection.

The algorithm for calculating the coordinates of a receptive field point by its ordinal number is described. It is shown that the time complexity of this algorithm can be estimated as ~$O(n^2)$, where $n$ is the space dimension. An outline of the algorithm for solving the linear programming problem by an artificial neural network using the constructed images is presented. A parallel algorithm for constructing an image of a linear programming problem on computing clusters is proposed. This algorithm is based on the BSF parallel computation model, which uses the master/workers paradigm and assumes a representation of the algorithm in the form of operations on lists using higher-order functions \textit{Map} and \textit{Reduce}. It is shown that the scalability bound of the parallel algorithm admits the estimation of $O(n\sqrt n)$. This means that the algorithm demonstrates a good scalability.

The parallel algorithm for constructing the multidimensional image of a linear programming problem is implemented in C++ using the BSF--skeleton that encapsulates all aspects related to parallelization by the MPI library and the OpenMP API. Using this software implementation, we conducted large-scale computational experiments on constructing images for random multidimensional linear programming problems with a large number of constraints on the ``Tornado SUSU'' computing cluster. The conducted experiments confirm the validity and efficiency of the proposed approaches. At the same time, it should be noted that the time of image construction increases exponentially with an increase in the space dimension. Therefore, the proposed method is applicable for problems with a number of variables not exceeding 100. However, the number of constraints can theoretically be unbounded.

Future research directions are as follows.
\begin{enumerate}
	\item Develop a method for solving linear programming problems based on the analysis of its images and prove its convergence.
	\item Develop and implement a method for the training data set generation to create a neural network that solves linear programming problems by analyzing their images.
	\item Develop and train an artificial neural network solving multidimensional linear programming problems.
	\item Develop and implement a parallel program on a computing cluster that constructs multidimensional images of a linear programming problem and calculates its solution using the artificial neural network.
\end{enumerate}

\section*{Funding} 
The reported study was partially funded by the Russian Foundation for Basic Research (project No.~20-07-00092-a) and the Ministry of Science and Higher Education of the Russian Federation (government order FENU-2020-0022).
\bibliographystyle{spmpsci}
\bibliography{Bibliography}
\end{document}